\documentclass[11pt, reqno]{article}

\usepackage[utf8]{inputenc}
\usepackage{lmodern}
\usepackage{subfiles}
\usepackage{enumitem}
	\setlist[enumerate]{label={\normalfont(\alph*)}} 
\usepackage{amsfonts}
\usepackage{amsthm}
\usepackage{amsmath}
\usepackage{amssymb}
\usepackage{amscd}
\usepackage{bbm}
\usepackage{multirow}
\usepackage{mathtools}
\usepackage{esint}
\usepackage[mathscr]{eucal}

\usepackage{titlefoot}
\usepackage[margin=3cm]{geometry}
\usepackage{indentfirst}
\usepackage{graphicx}
\usepackage{graphics}
\usepackage{lscape}
\usepackage{tikz-cd}
\usepackage{color}
\usepackage{pict2e}
\usepackage{epic}
\usepackage{epstopdf}
\usepackage{titlesec}
	\titleformat{\section}[block]{\Large\bfseries\filcenter}{\thesection}{1em}{}
\usepackage{commath}
\usepackage{float}
\usepackage{caption}
\usepackage{subcaption}
\usepackage{etoolbox}
\usepackage[affil-it]{authblk}
\usepackage{combelow}

\let\oldbibliography\thebibliography
\renewcommand{\thebibliography}[1]{%
  \oldbibliography{#1}%
  \setlength{\itemsep}{-.5pt}%
}

\usepackage[hidelinks]{hyperref}
\hypersetup{bookmarksopen=true} 
\usepackage{hypcap}

\graphicspath{{./Pictures/}}
\allowdisplaybreaks

\usepackage{algorithm}
\usepackage[noend]{algpseudocode}

\usepackage{listings}
\usepackage{fancybox}

\renewcommand*\thesection{\arabic{section}}

\theoremstyle{plain}
\newtheorem{teo}[equation]{Theorem}
\newtheorem{lema}[equation]{Lemma}

\newtheorem{cor}[equation]{Corollary}

\theoremstyle{definition}

\newtheorem{remark}[equation]{Remark}

\newcommand{\thistheoremname}{}
\newtheorem{genericthm}[equation]{\thistheoremname}

\newcommand{\thistheoremnames}{}
\newtheorem*{genericthms}{\thistheoremnames}
\newenvironment{para*}[1]
  {\renewcommand{\thistheoremnames}{#1}%
   \begin{genericthms}}
  {\end{genericthms}}

\expandafter\let\expandafter\oldproof\csname\string\proof\endcsname
\let\oldendproof\endproof
\renewenvironment{proof}[1][\proofname]{%
  \oldproof[\upshape \bfseries #1:]%
}{\oldendproof}

\makeatletter
\def\@makechapterhead#1{%
  \vspace*{50\p@}%
  {\parindent \z@ \raggedright \normalfont
    \interlinepenalty\@M
    \Huge\bfseries  \thechapter.\quad #1\par\nobreak
    \vskip 40\p@
  }}
\makeatother

\def \R {\mathbb{R}}

\def \D{\textup{D}}

\def \d{\,\textup{d}}

\def \p{\partial}

\def \mc{\mathcal}

\def \tp{\textup}

\makeatletter
\DeclareFontFamily{OMX}{MnSymbolE}{}
\DeclareSymbolFont{MnLargeSymbols}{OMX}{MnSymbolE}{m}{n}
\SetSymbolFont{MnLargeSymbols}{bold}{OMX}{MnSymbolE}{b}{n}
\DeclareFontShape{OMX}{MnSymbolE}{m}{n}{
    <-6>  MnSymbolE5
   <6-7>  MnSymbolE6
   <7-8>  MnSymbolE7
   <8-9>  MnSymbolE8
   <9-10> MnSymbolE9
  <10-12> MnSymbolE10
  <12->   MnSymbolE12
}{}
\DeclareFontShape{OMX}{MnSymbolE}{b}{n}{
    <-6>  MnSymbolE-Bold5
   <6-7>  MnSymbolE-Bold6
   <7-8>  MnSymbolE-Bold7
   <8-9>  MnSymbolE-Bold8
   <9-10> MnSymbolE-Bold9
  <10-12> MnSymbolE-Bold10
  <12->   MnSymbolE-Bold12
}{}

\let\llangle\@undefined
\let\rrangle\@undefined
\DeclareMathDelimiter{\llangle}{\mathopen}%
                     {MnLargeSymbols}{'164}{MnLargeSymbols}{'164}
\DeclareMathDelimiter{\rrangle}{\mathclose}%
                     {MnLargeSymbols}{'171}{MnLargeSymbols}{'171}
\makeatother

\begin{document}

 \title{\LARGE \textbf{Remarks on Ornstein's non-inequality in
     $\R^{2\times 2}$}}

\author[1]{{\Large Daniel Faraco}}
\author[2]{{\Large Andr\'e Guerra\vspace{0.4cm}}}

\affil[1]{\small 
Universidad Autónoma de Madrid,  Departamento de Matemáticas, 28049 Madrid, Spain  \protect \\
{\tt{daniel.faraco@uam.es}}\vspace{1em}\ }

\affil[2]{\small University of Oxford, Andrew Wiles Building, Woodstock Rd, Oxford OX2 6GG, United Kingdom \protect \\
{\tt{andre.guerra@maths.ox.ac.uk}}\ }

\date{}

\maketitle

\begin{abstract}
We give a very concise proof of Ornstein's $L^1$ non-inequality for first- and second-order operators in two dimensions. The proof just needs a two-dimensional laminate supported on three points.
\end{abstract}

\vspace{0.2cm}

\unmarkedfntext{
\hspace{-0.85cm} 


\noindent  \emph{Acknowledgments:} The authors thank Jan Kristensen for interesting comments. D.F.\ acknowledges financial support from the Spanish Ministry of Science and Innovation through the Severo Ochoa Programme for Centres of Excellence in R\&D (CEX2019-000904-S) and the ICMAT Severo Ochoa grant SEV2015-0554 and 
 is partially supported by the Line of excellence for University Teaching Staff between CM and UAM, by the ERC Advanced Grant 834728 and
by the MTM2017-85934-C3-2-P.
A.G. was supported by  [EP/L015811/1].
}

Given two linear constant-coefficient homogeneous $k$-th order differential operators $\mc P_1$, $\mc P_2$ and a number $1\leq p\leq \infty$, consider the inequality
\begin{equation}
\label{eq:genestimate}
\Vert \mc P_1 \varphi \Vert_{L^p(\R^n)} \leq C_p \Vert \mc P_2 \varphi \Vert_{L^p(\R^n)}, \qquad \tp{for all }\varphi \in C^\infty_c(\R^n,\R^m),
\end{equation}
where $0<C_p$ is some constant.
When does such an estimate hold?

The case $1<p<\infty$ is very classical: if we take $\mc P_1=\D^k$ to be the $k$-th order gradient then \eqref{eq:genestimate} holds if and only if $\mc P_2$ is an elliptic operator (in the sense that it has injective symbol); this is a classical result which goes back to the work of Calderón and Zygmund \cite{Calderon1952}. We refer the reader to \cite{GuerraRaita2020} for a generalisation of \eqref{eq:genestimate} which allows for operators with non-trivial kernels.

At the end-points $p=1$ or $p=\infty$, \eqref{eq:genestimate} never holds, except in trivial situations: this was proved, respectively, by Ornstein \cite{Ornstein1962} and Mityagin \cite{Mityagin1958}, but see also \cite{deLeeuw1962} for the $p=\infty$ case. In some circumstances one can deduce the result for $p=\infty$ from the one for $p=1$, see for instance \cite{Bourgain2002, McMullen1998} for the case $\mc P_2=\tp{div}$, and in fact the result for $p=1$ is much more difficult. Similar results also hold in the anisotropic setting, see \cite{Kazaniecki2015, Prosinski2020}.

The failure of \eqref{eq:genestimate} when $p=1$ can also be deduced from the Kirchheim--Kristensen convexity theorem \cite{Kirchheim2011, Kirchheim2016}. Besides providing a concise proof of Ornstein's result, their theorem also has applications to
the regularity of hessians of rank-one convex functions and to the
characterization of gradient Young measures \cite{Kristensen2019, Kristensen2010}. Both the failure of \eqref{eq:genestimate} when $p=1$, $\mc P_1=\D, \mc P_2=\mc E$, as well as the existence of rank-one convex functions on 3-dimensional spaces with irregular hessians, were proved by the first author and collaborators in
\cite{Conti2005, Conti2005a} through constructions with unbounded  laminates (the so-called staircase laminates) introduced in \cite{Faraco2003b,Faraco2004}. See also \cite{Astala2008} and \cite{Faraco2018, Oliva2016} for related problems for $p>1$.  Such laminates can
be used to provide a fairly explicit deformation showing Ornstein's non inequality, but their construction is somewhat complicated as it includes an infinite process.

The purpose of this note is to give an alternative proof  of Ornstein's result when $n=m=2$ and $k=1$. These assumptions encompass the case where $\mc P_1=\D$ and $\mc P_2$ is any of the operators
$$
(\tp{div}, \tp{curl}),
\qquad \mc E u\equiv \frac 1 2(\D u +(\D u)^\tp{T}),
\qquad \bar \p u\equiv \mc E u-\frac{\tp{div}\,u}{2} \tp{Id}_2,
$$
which appear respectively in electromagnetism, linearized elasticity and complex analysis.
Our strategy is similar to the one of Kirchheim--Kristensen \cite{Kirchheim2011, Kirchheim2016} and in fact we prove a particular case of their convexity theorem. However, our approach is less elaborate, as the  singular value decomposition reduces the problem to building a laminate on the diagonal matrices and, since the integrand of interest is 1-homogeneous, the laminate is very simple. 
An interesting question that we do not address here is whether, at least for integrands with symmetries, there exists a Kirchheim--Kristensen theory for $p>1$.



For a bounded open set $\Omega\subset \R^n$ and a function $f\colon \R^{m\times n}\to \R$, we recall that:
\begin{enumerate}[topsep=1pt]
  \itemsep0em 
\item\label{it:qc} $f$ is \emph{quasiconvex} at $A\in \R^{m\times n}$ if
  $0\leq \int_{\Omega} f(A+\D\varphi)-f(A)\d x$ for all $\varphi \in C^\infty_c(\Omega,\R^m)$;
\item\label{it:rc} $f$ is \emph{rank-one
    convex} if $t\mapsto f(A+tX)$ is convex for all
  $A, X\in \R^{m\times n}$ with $\tp{rank}\,X=1$;
\item $f$
  is \emph{positively 1-homogeneous} if $f(t A)=t f(A)$ for all $t>0$
  and all $A\in \R^{m\times n}$;
  \item $f$ is \emph{1-homogeneous} if $f(t A)=|t| f(A)$ for all $t\in \R$ and all $A\in \R^{m\times n}$.
\end{enumerate}
It is well-known that the definition in \ref{it:qc} is independent of $\Omega$ and that \ref{it:qc} $\Rightarrow$ \ref{it:rc}, see \cite{Dacorogna2007}, although in general the converse is not true \cite{Sverak1992a}. 
In \cite{Kirchheim2016}, the following theorem was proved: 

\begin{teo}[Kirchheim--Kristensen]\label{thm:KK}
  Let $f\colon \R^{m\times n} \to \R$ be positively 1-homogeneous and
  rank-one convex. Then $f$ is convex at all matrices $X$ with
  $\tp{rank}\, X\leq 1$.
\end{teo}

The reader may find other results concerning positively 1-homogeneous rank-one convex functions in \cite{Dacorogna2008, Muller1992, Sverak1991}.
In the planar case there is a particularly simple proof of Theorem \ref{thm:KK} for  1-homogeneous functions:

\begin{lema}\label{lema:laminate}
  Let $f\colon \R^{2\times 2}\to \R$ be 1-homogeneous and
  rank-one convex. 
  Then $f\geq 0$ and moreover, as $f(0)=0$, $f$ is convex at zero.
\end{lema}

\begin{proof}
For $A\in \R^{2\times 2}$ the singular value decomposition yields $Q,R\in \tp{O}(2)$ and
  $\Lambda\in \R^{2\times 2}_\tp{diag}$ such that $A=Q\Lambda R;$
  moreover, the entries of $\Lambda$ are non-negative. Let us write
  $(x,y)\equiv \tp{diag}(x,y)$. If $0\neq A$ then, by
  homogeneity, we can
  assume that $\Lambda=(1,y)$, where $y\geq 0$.  The measure
$$\nu = \frac 1 2 \delta_{Q(2,-2y)R}+\frac 1 3
\delta_{Q(1,y)R} + \frac 1 6 \delta_{Q(-2,-2 y)R}$$ is a laminate with
barycentre $Q(1,-y)R$. Indeed, we have the splittings
$$
(1,-y)\to \frac 1 3 (1,y) + \frac 2 3 (1,-2y) \to \frac 1 3(1,y) + \frac 1 6 (-2,-2 y) + \frac 1 2 (2,-2 y)
$$
and the map $A\mapsto QAR$ is rank-preserving.
Since $f$ is rank-one convex and 1-homogeneous,
\begin{align*}
  f\left(Q(1,-y)R\right) &\leq \frac 1 2 f\left(Q(2,-2y)R\right) + \frac 1 3 f\left(Q(1,y)R\right) + \frac 1 6
                f\left(Q(-2,-2y)R\right)\\
              & = f\left(Q(1,-y)R\right) + \frac 1 3 f\left(A\right) + \frac 1 3 f\left(-A\right)
\end{align*}
Hence $0\leq f(A)+f(-A)=2f(A)$ and
the proof is finished.
\end{proof}

\begin{remark}\label{remark:sym}
An identical proof gives the same conclusion if $f\colon \R^{2\times
    2}_\tp{sym}\to \R$; in this case one takes $R=Q^\tp{T}$, since symmetric
  matrices are diagonalisable by orthogonal matrices.
\end{remark}

From Lemma \ref{lema:laminate} we get a two-dimensional
version of Ornstein's non-inequality:

\begin{teo}\label{teo:main}
  Let $\Omega\subset \R^2$ be a bounded open set and let $\mc P_i$ be first-order differential operators, $i=1,2$, acting on $\varphi\in C^\infty_c(\Omega,\R^2)$ by
  $\mc P_i \varphi=P_i (\D \varphi)$, where $P_i\in \tp{Lin}(\R^{2\times 2},\R^{d_i})$.
  
Suppose that there is a constant $C$ such that
\begin{equation}
\Vert \mc P_1 \varphi \Vert_{L^1}\leq C
\Vert\mc P_2 \varphi \Vert_{L^1}
\qquad \tp{ for all }\varphi \in
C^\infty_c(\Omega,\R^2).\label{eq:estimate}
\end{equation}
Then there is $T\in \tp{Lin}(\R^{d_2},\R^{d_1})$ such that $P_1=T \circ P_2$.
Moreover, the same conclusion is true if we require that \eqref{eq:estimate} holds only for those $\varphi$ of the form $\varphi = \nabla \phi$ for some $\phi \in C^\infty_c(\Omega,\R)$.
\end{teo}

\begin{proof}
Consider the function $f\colon \R^{2\times 2}\to \R$ defined by
$f(A)=C\norm{P_2A}-  \norm{P_1 A}$. Its quasiconvex envelope
$f^\tp{qc}\colon \R^{2\times 2}\to [-\infty,\infty)$ is given
by the Dacorogna formula
$$f^\tp{qc}(A)=\inf_{\varphi \in C^\infty_c(\R^2,\R^2)} \int_{\R^2}
f(A+\D \varphi) \d x;$$
it is easily checked that $f^\tp{qc}$
is 1-homogeneous, since the same
holds for $f$. Note that (\ref{eq:estimate}) is equivalent to
$f^\tp{qc}(0)\geq 0$; thus $f^\tp{qc}>-\infty$ everywhere and hence
$f^\tp{qc}$ is rank-one convex. Applying Lemma \ref{lema:laminate} we
see that $0\leq f^\tp{qc}\leq f$ and so we must have $\ker
P_2\subseteq \ker P_1$. Take $T=P_1 P_2^\dagger$, where
$P_2^\dagger$ is the \emph{Moore--Penrose inverse}, defined by
$$P_2^\dagger\equiv \left( P_2|_{(\ker P_2)^\bot} \right)^{-1}
\tp{Proj}_{\tp{im\,}P_2}.$$
Since $P_2^\dagger P_2$ is the orthogonal projection onto $(\ker
P_2)^\bot$, the conclusion follows.

The last part is identical, except that we replace Lemma \ref{lema:laminate} with Remark \ref{remark:sym}: if \eqref{eq:estimate} holds for all potential vector fields then $(f|_{\R^{2\times 2}_\tp{sym}})^\tp{qc}(0)\geq 0$, see  \cite{Ball1981,Sverak1992} for quasiconvexity on $\R^{n\times n}_\tp{sym}$.
\end{proof}

In particular, from the second part of Theorem \ref{teo:main} we recover \cite[Part 1]{Ornstein1962}:

\begin{cor}
Given a bounded open set $\Omega\subset \R^2$, there is no constant $C$ such that
$$\int_{\Omega} \left|\p_{x_1 x_2}\phi (x) \right|\d x \leq C
\int_{\Omega} \left|\p_{x_1x_1} \phi(x) \right|+\left|\p_{x_2x_2}\phi(x)\right|\d x 
\qquad
\tp{ for all } \phi \in C^\infty_c(\Omega).
$$
\end{cor}

{\footnotesize
\bibliographystyle{acm}

\bibliography{/Users/antonialopes/Dropbox/Oxford/Bibtex/library.bib}
}

\end{document}